\documentclass[12pt]{article}
\usepackage{e-jc}

\usepackage{amsmath,amssymb}
\usepackage{amsthm}
\newtheorem{my_conjecture}[theorem]{\bf Conjecture/Question}

\usepackage{graphicx}
\usepackage[unicode=true,
 bookmarks=false,
 breaklinks=false,pdfborder={0 0 1},backref=true,colorlinks=true]
 {hyperref}
\hypersetup{
 urlcolor=blue,linkcolor=blue,citecolor=blue}
\allowdisplaybreaks
\allowbreak

\dateline{Apr 17, 2017}{May 22, 2018}{TBD}


\Copyright{  The authors. Released under the CC BY-ND license (International 4.0).}

\title{A Note on the Expected Length of the Longest Common Subsequences of two i.i.d.~Random Permutations}

\author{Christian Houdr\'{e}\thanks{Research supported in part by the grants \# 246283 and \# 524678 from the Simons Foundation.}\\
\small School of Mathematics\\[-0.8ex]
\small Georgia Institute of Technology\\[-0.8ex] 
\small Atlanta, Georgia, U.S.A.\\
\small\tt houdre@math.gatech.edu\\
\and
Chen Xu \\
\small School of Mathematics\\[-0.8ex]
\small Georgia Institute of Technology\\[-0.8ex]
\small Atlanta, Georgia, U.S.A.\\
\small\tt cxu60@gatech.edu}

\begin{document}

\maketitle


\begin{abstract}
We address a question and a conjecture on the expected length of the
longest common subsequences of two i.i.d.$\ $random permutations
of $[n]:=\{1,2,...,n\}$. The question is resolved by showing that
the minimal expectation is not attained in the uniform case. The conjecture
asserts that $\sqrt{n}$ is a lower bound on this expectation, but
we only obtain $\sqrt[3]{n}$ for it.
\end{abstract}

\section{Introduction}

The length of the longest increasing subsequences ($\mathit{LISs}$)
of a uniform random permutation $\sigma\in{\cal S}_{n}$ (where ${\cal S}_{n}$
is the symmetric group) is well studied and we refer to the monograph
\cite{romik2015surprising} for precise results and a comprehensive
bibliography on this subject. Recently, \cite{houdre2014central}
showed that for two independent random permutations $\sigma_{1},\ \sigma_{2}\in{\cal S}_{n}$,
and as long as $\sigma_{1}$ is uniformly distributed and regardless
of the distribution of $\sigma_{2}$, the length of the longest common
subsequences ($LCSs$) of the two permutations is identical in law
to the length of the $LISs$ of $\sigma_{1}$, i.e. $LCS(\sigma_{1},\sigma_{2})=^{\mathcal{L}}LIS(\sigma_{1})$.
This equality ensures, in particular, that when $\sigma_{1}$ and
$\sigma_{2}$ are uniformly distributed, $\mathbb{E}LCS(\sigma_{1},\sigma_{2})$
is upper bounded by $2\sqrt{n}$, for any $n$, (see \cite{pilpel1990descending})
and asymptotically of order $2\sqrt{n}$ (\cite{romik2015surprising}).
It is then rather natural to study the behavior of $LCS(\sigma_{1},\sigma_{2})$,
when $\sigma_{1}$ and $\sigma_{2}$ are i.i.d.$\ $but not necessarily
uniform. In this respect, Bukh and Zhou raised, in \cite{bukh2016twins},
two issues which can be rephrased as follows:
\begin{my_conjecture}
Let $P$ be an arbitrary probability distribution on ${\cal S}_{n}$.
Let $\sigma_{1}$ and $\sigma_{2}$ be two i.i.d. permutations sampled
from $P$. Then $\mathbb{E}_{P}[LCS(\sigma_{1},\sigma_{2})]\ge\sqrt{n}$.
It might even be true that the uniform distribution $U$ on ${\cal S}_{n}$
gives a minimizer.
\end{my_conjecture}
Below we prove the suboptimality of the uniform distribution by explicitly
building a distribution having a  smaller expectation. In the next
section, before presenting and proving our main result, we give a
few definitions and formalize this minimizing problem as a quadratic
programming one. Section \ref{sec:propertyl} further explore some
properties of the spectrum of the coefficient matrix of our quadratic
program. In the concluding section, a quick cubic root lower bound
is given along with a few pointers for future research.

\section{Main Results\label{Sec:main}}

We begin with a few notations. Throughout, $\sigma$ and $\pi$ are,
respectively, used for random and deterministic permutations. By convention,
$[n]:=\{1,2,3,...,n\}$ and so $\{\pi_{i}\}_{i\in[n!]}={\cal S}_{n}$
is a particular ordered enumeration of ${\cal S}_{n}$. (Some other
orderings of ${\cal S}_{n}$ will be given when necessary.) Next,
a random permutation $\sigma$ is said to be sampled from $P=(p_{i})_{i\in[n!]}$,
if $\mathbb{P}_{P}(\sigma=\pi_{i})=p_{i}$. The uniform distribution
is therefore $U=(1/n!)_{i\in[n!]}$ and, for simplification, it is
denoted by $E/n!$, where $E=(1)_{i\in[n!]}$ is the $n$-tuple only
made up of ones. When needed, a superscript will indicate the degree
of the symmetric group we are studying, e.g., $\sigma^{(n)}$ and
$P^{(n)}$ are respectively a random permutation and distribution
from ${\cal S}_{n}$. 

Let us now formalize the expectation as a quadratic form:
\begin{align}
\mathbb{E}_{P}[LCS(\sigma_{1},\sigma_{2})] & =\sum_{i,j\in[n!]}p_{i}LCS(\pi_{i},\pi_{j})p_{j}\nonumber \\
 & =\sum_{i,j\in[n!]}p_{i}\ell_{ij}p_{j}=P^{T}L^{(n)}P,\label{eq:quadoptform}
\end{align}
where $\ell_{ij}:=LCS(\pi_{i},\pi_{j})$ and $L^{(n)}:=\{\ell_{ij}\}_{(i,j)\in[n!]\times[n!]}$.
It is clear that $\ell_{ij}=\ell_{ji}$ and that $\ell_{ii}=n$. A
quick analysis of the cases $n=2$ or $3$ shows that both $L^{(2)}$
and $L^{(3)}$ are positive semi-definite. However, this property
does not hold further:
\begin{lemma}
\label{lem:nondef}For $n\ge4,$ the smallest eigenvalue $\lambda_{1}^{(n)}$
of $L^{(n)}$ is negative.
\end{lemma}

\begin{proof}
Linear algebra gives $\lambda_{1}^{(2)}=1$ and $\lambda_{1}^{(3)}=0$.
So to prove the result, it suffices to show that $\lambda_{1}^{(k+1)}<\lambda_{1}^{(k)}$,
$k\ge1$ and this is done by induction. The base case is true, since
$\lambda_{1}^{(2)}=1>0=\lambda_{1}^{(3)}$. To reveal the connection
between $L^{(k+1)}$ and $L^{(k)}$, the enumeration of ${\cal S}_{k+1}$
is iteratively built on that of ${\cal S}_{k}$ by inserting the new
element $(k+1)$ into the permutations from ${\cal S}_{k}$ in the
following way: the enumeration of the $(k+1)!$ permutations is split
into $(k+1)$ trunks of equal size $k!$. In the $ith$ trunk, the
new element $(k+1)$ is inserted behind the $(k+1-i)th$ digit in
the permutation from ${\cal S}_{k}$. (For example, if ${\cal S}_{2}$
is enumerated as $\{[12],[21]\}$, then the enumeration of the first
trunk in ${\cal S}_{3}$ is $\{[123],[213]\}$, the second is $\{[132],[231]\}$
and the third is $\{[312],[321]\}$. Then the overall enumeration
for ${\cal S}_{3}$ is $\{[123],[213],[132],[231],[312],[321]\}$.)

Via this enumeration, the principal minor of size $k!\times k!$ is
row and column indexed by the enumeration of the permutations $\{\pi_{i}^{(k)}\}_{i\in[k!]}$
from ${\cal S}_{k}$ with $(k+1)$ as the last digit, i.e., $\{[\pi_{i}^{(k)}(k+1)]\}_{i\in[k!]}\subseteq{\cal S}_{k+1}$.
Then the $(i,j)$ entry of the submatrix is 
\[
LCS([\pi_{i}(k+1)],[\pi_{j}(k+1)])=LCS(\pi_{i},\pi_{j})+1,
\]
since the last digit $(k+1)$ adds an extra element into the longest
common subsequences. Hence, the $k!\times k!$ principal minor of
$L^{(k+1)}$ is $L^{(k)}+E^{(k)}(E^{(k)})^{T}$, where $E^{(k)}$
is the vector of $\mathbb{R}^{k!}$ only made up of ones. Moreover,
notice that the sum of the $\pi_{i}$-indexed row of $L^{(k)}$ is
\begin{align*}
\sum_{j\in[k!]}LCS(\pi_{i},\pi_{j}) & =\sum_{j\in[k!]}LCS(id,\pi_{i}^{-1}\pi_{j})\\
 & =\sum_{j\in[k!]}LIS(\pi_{i}^{-1}\pi_{j}),
\end{align*}
since simultaneously relabeling $\pi_{i}$ and $\pi_{j}$ does not
change the length of the $LCSs$ and also since a particular relabeling
to make $\pi_{i}$ to be the identity permutation, which is equivalent
to left composition by $\pi_{i}^{-1}$, is applied here. Further,
any $LCS$ of the identity permutation and of $\pi_{i}^{-1}\pi_{j}$
is a $LIS$ of $\pi_{i}^{-1}\pi_{j}$ and vice versa. So the row sum
is equal to
\[
\sum_{j\in[k!]}LIS(\pi_{i}^{-1}\pi_{j})=\sum_{\pi\in{\cal S}_{k}}LIS(\pi),
\]
since left composition by $\pi_{i}^{-1}$ is a bijection from ${\cal S}_{k}$
to ${\cal S}_{k}$. This indicates that all the row sums of $L^{(k)}$
are equal. Hence, $E^{(k)}$ is actually a right eigenvector of $L^{(k)}$
and is associated with the row sum $\sum_{\pi\in{\cal S}_{k}}LIS(\pi)>0$
as its eigenvalue, which is distinct from the smallest eigenvalue
$\lambda_{1}^{(k)}\le0$.

On the other hand, since $L^{(k)}$ is symmetric, the eigenvectors
$R_{1}^{(k)}$ and $E^{(k)}$ associated with the eigenvalues $\lambda_{1}^{(k)}$
and $\sum_{\pi\in{\cal S}_{k}}LIS(\pi)$ are orthogonal, i.e.,
\begin{equation}
(E^{(k)})^{T}R_{1}^{(k)}=0.\label{eq:eigenperpend}
\end{equation}
Without loss of generality, let $R_{1}^{(k)}$ be a unit vector, then
from (\ref{eq:eigenperpend}),
\begin{align}
\lambda_{1}^{(k)} & =(R_{1}^{(k)})^{T}L^{(k)}(R_{1}^{(k)})\nonumber \\
 & =(R_{1}^{(k)})^{T}(L^{(k)}+E^{(k)}(E^{(k)})^{T})R_{1}^{(k)}.\label{eq:quad}
\end{align}
As $L^{(k)}+E^{(k)}(E^{(k)})^{T}$ is the $k!\times k!$ principal
minor of $L^{(k+1)}$, (\ref{eq:quad}) becomes
\begin{equation}
\left[\begin{array}{c}
R_{1}^{(k)}\\
0
\end{array}\right]^{T}L^{(k+1)}\left[\begin{array}{c}
R_{1}^{(k)}\\
0
\end{array}\right]\ge\min_{R^{T}E=0,||R||=1}R^{T}L^{(k+1)}R=\lambda_{1}^{(k+1)},\label{eq:inductionlbd}
\end{equation}
where $R_{1}^{(k)}$ is properly extended to $\left[\begin{array}{c}
R_{1}^{(k)}\\
0
\end{array}\right]\in\mathbb{R}^{(k+1)!}$ and where the above inequality holds true since $\left[\begin{array}{c}
R_{1}^{(k)}\\
0
\end{array}\right]^{T}E^{(k+1)}=(R_{1}^{(k)})^{T}E^{(k)}=0$ and $\left\Vert \left[\begin{array}{c}
R_{1}^{(k)}\\
0
\end{array}\right]\right\Vert =\left\Vert R_{1}^{(k)}\right\Vert =1$, where $\left\Vert \cdot\right\Vert $ denotes the corresponding
Euclidean norm. Moreover, equality in (\ref{eq:inductionlbd}) holds
if and only if $\left[\begin{array}{c}
R_{1}^{(k)}\\
0
\end{array}\right]$ is a eigenvector of $L^{(k+1)}$ associated with $\lambda_{1}^{(k+1)}$.
We show next, by contradiction, that this cannot be the case. Indeed,
assume that
\begin{equation}
L^{(k+1)}\left[\begin{array}{c}
R_{1}^{(k)}\\
0
\end{array}\right]=\lambda_{1}^{(k+1)}\left[\begin{array}{c}
R_{1}^{(k)}\\
0
\end{array}\right].\label{eq:eigendef}
\end{equation}
Now, consider the $k!\times k!$ submatrix at the bottom-left corner
of $L^{(k+1)}$, which is row-indexed by $\{[(k+1)\pi_{i}]\}_{i\in[k!]}$
and column-indexed by $\{[\pi_{i}(k+1)]\}_{i\in[k!]}$. Notice that
the $(i,j)$-entry of this submatrix is 
\[
LCS([(k+1)\pi_{i}],[\pi_{j}(k+1)])=LCS(\pi_{i},\pi_{j}),
\]
since $(k+1)$ can be in some $LCS$ only if the length of this $LCS$
is $1$. So this submatrix is in fact equal to $L^{(k)}$. Further,
the vector consisting of the bottom $k!$ elements on the left-hand-side
of (\ref{eq:eigendef}) is $L^{(k)}R_{1}^{(k)}=\lambda_{1}^{(k)}R_{1}^{(k)}$,
which is a non-zero vector. However, on the right-hand-side, the corresponding
bottom $k!$ elements of the vector $\left[\begin{array}{c}
R_{1}^{(k)}\\
0
\end{array}\right]$ form the zero vector. This leads to a contradiction\@. So,
\[
\lambda_{1}^{(2)}=1>0=\lambda_{1}^{(3)}>\lambda_{1}^{(4)}>\lambda_{1}^{(5)}...
\]
\end{proof}
The above result on the smallest negative eigenvalue, and its associated
eigenvector, will help build a distribution on ${\cal S}_{n}$, for
which the $LCSs$ have a smaller expectation than for the uniform
one.
\begin{theorem}
Let $\sigma_{1}$ and $\sigma_{2}$ be two i.i.d.~random permutations
sampled from a distribution $P$ on the symmetric group ${\cal S}_{n}$.
Then, for $n\le3$, the uniform distribution $U$ minimizes $\mathbb{E}_{p}[LCS(\sigma_{1},\sigma_{2})]$,
while, for $n\ge4$, $U$ is sub-optimal.
\end{theorem}

\begin{proof}
As we have seen in (\ref{eq:quadoptform}),
\begin{align}
\mathbb{E}_{P}[LCS(\sigma_{1},\sigma_{2})] & =P^{T}LP\nonumber \\
 & =(P-U)^{T}L(P-U)+2P^{T}LU-U^{T}LU\nonumber \\
 & =(P-U)^{T}L(P-U)+2U^{T}LU-U^{T}LU\nonumber \\
 & =(P-U)^{T}L(P-U)+U^{T}LU,\label{eq:quadformU}
\end{align}
where $P^{T}LU=U^{T}LU$, since $U$ is an eigenvector of $L$ and
$P^{T}U=1$.

When $n=2,3$, $L^{(n)}$ is positive semi-definite and therefore
$(P-U)^{T}L(P-U)\ge0$. So, $P^{T}LP\ge U^{T}LU$.

However, when $n\ge4$, by Lemma \ref{lem:nondef}, the smallest eigenvalue
$\lambda_{1}^{(n)}$ is strictly negative and the associated eigenvector
$R_{1}^{(n)}$ is such that $U^{T}R_{1}^{(n)}=0=E^{T}R_{1}^{(n)}$.
Hence, there exists a positive constant $c$ such that $cR_{1}^{(n)}\succeq-1/n!$,
where $\succeq$ stands for componentwise inequality. Let $P_{0}$
be such that $P_{0}-U=cR_{1}^{(n)}$, then it is immediate that 
\[
E^{T}P_{0}=E^{T}(U+cR_{1}^{(n)})=1+0=1,
\]
and that
\[
P_{0}=U+cR_{1}^{(n)}\succeq0.
\]
Therefore, $P_{0}$ is a well-defined distribution on ${\cal S}_{n}$.
On the other hand, by (\ref{eq:quadformU}), the expectation under
$P_{0}$ is such that
\begin{align}
\mathbb{E}_{P_{0}}[LCS(\sigma_{1},\sigma_{2})] & =(P_{0}-U)^{T}L(P_{0}-U)+U^{T}LU\nonumber \\
 & =c^{2}(R_{1}^{(n)})^{T}L(R_{1}^{(n)})+U^{T}LU\nonumber \\
 & =c^{2}\lambda_{1}^{(n)}+U^{T}LU\nonumber \\
 & <U^{T}LU.\label{eq:ULUbound}
\end{align}
However, the right-hand side of (\ref{eq:ULUbound}) is nothing but
the expectation under the uniform distribution, namely, $\mathbb{E}_{U}[LCS(\sigma_{1},\sigma_{2})]$.
\end{proof}
The existence of negative eigenvalues contributes to the above construction
and to the corresponding counterexample. So, as a next step, properties
of this smallest negative eigenvalue and of the spectrum of the coefficient
matrix $L^{(n)}$ are explored.

\section{Further Properties of $L^{(n)}$\label{sec:propertyl}}

As we have seen, the vector $E^{(n)}$ which is made up of only ones
is an eigenvector associated with the eigenvalue $\sum_{\pi\in S_{n}}LIS(\pi)$.
It is not hard to show that this eigenvalue is, in fact, the spectral
radius of $L^{(n)}$.
\begin{proposition}
$\sum_{\pi\in S_{k}}LIS(\pi)$ is the spectral radius of $L^{(n)}$.
\end{proposition}

\begin{proof}
Without loss of generality, let $(\lambda,R)$ be a pair of eigenvalue
and corresponding eigenvector of $L^{(n)}$ such that $\max_{i\in[n!]}\vert r_{i}\vert=1$,
where $R=(r_{1},...,r_{n!})^{T}$, and let $i_{0}$ be the index such
that $\vert r_{i_{0}}\vert=1$. Let us focus now on the $i_{0}th$
element of $\lambda R$. Then, since $L^{(n)}R=\lambda R$, 
\begin{eqnarray*}
|\lambda| & = & |\lambda r_{i_{0}}|\\
 & = & \left|\sum_{j\in[n!]}LCS(\pi_{i_{0}},\pi_{j})r_{j}\right|\\
 & \le & \sum_{j\in[n!]}LCS(\pi_{i_{0}},\pi_{j})\\
 & = & \sum_{j\in[n!]}LIS(\pi_{i_{0}}^{-1}\pi_{j})\\
 & = & \sum_{\pi\in S_{n}}LIS(\pi),
\end{eqnarray*}
with equality if and only if all the $r_{j}$'s have the same sign
and have absolute value equal to $1$.
\end{proof}
This gives a trivial bound on the smallest negative value $\lambda_{1}^{(n)}$:
namely, $$\lambda_{1}^{(n)}\ge \allowbreak -\sum_{\pi\in S_{n}}LIS(\pi).$$ Moreover,
since the expectation of the longest increasing subsequence of a uniform
random permutation is asymptotically $2\sqrt{n}$, this gives an asymptotic
order of $-2n!\sqrt{n}$ for the lower bound. On the other hand, we
are interested in an upper bound for $\lambda_{1}^{(n)}$. The next
result shows that $\lambda_{1}^{(n)}$ decreases at least exponentially
fast, in $n$.
\begin{proposition}
$\lambda_{1}^{(n)}\le2^{n-4}\lambda_{1}^{(4)}=-2^{n-3}<0$.
\end{proposition}

\begin{proof}
\begin{flushleft}
This is proved by showing that $\lambda_{1}^{(n+1)}\le2\lambda_{1}^{(n)}$.
As well known,
\begin{equation}
\lambda_{1}^{(n+1)}=\min_{E^{T}R=0}\frac{R^{T}L^{(n+1)}R}{R^{T}R}.\label{eq:smallesteigenval}
\end{equation}
Let $\lambda_{1}^{(n)}$ be the smallest eigenvalues of $L^{(n)}$
and let $R^{(n)}$ be the corresponding eigenvector. Then, in generating
$L^{(n+1)}$ from $L^{(n)}$ as done in the proof of Lemma \ref{lem:nondef},
the $n!\times n!$ principal minor of $L^{(n+1)}$ is $L^{(n)}+EE^{T}$,
while its bottom-left $n!\times n!$ submatrix is $L^{(n)}$. Symmetrically,
it can be proved that the top-right $n!\times n!$ submatrix is also
$L^{(n)}$, while the bottom-right $n!\times n!$ submatrix is $L^{(n)}+EE^{T}$,
i.e., $L^{(n+1)}$ is 
\[
\left[\begin{array}{ccc}
L^{(n)}+EE^{T} & \cdots & L^{(n)}\\
\vdots & \ddots & \vdots\\
L^{(n)} & \cdots & L^{(n)}+EE^{T}
\end{array}\right].
\]
Further, let
\[
R=\left[\begin{array}{c}
R_{1}^{(n)}\\
0\\
\vdots\\
0\\
R_{1}^{(n)}
\end{array}\right].
\]
Then $E^{T}R=E^{T}R_{1}^{(n)}+E^{T}R_{1}^{(n)}=0$, where, by an abuse
of notation, $E$ denotes the vector only made up of ones and of the
appropriate dimension. Also,
\begin{align*}
\left\Vert R\right\Vert ^{2} & =R^{T}R=2\left\Vert R_{1}^{(n)}\right\Vert ^{2}=2.
\end{align*}
In (\ref{eq:smallesteigenval}), the corresponding numerator $R^{T}L^{(n+1)}R$
is 
\begin{align*}
\left[\begin{array}{c}
R_{1}^{(n)}\\
0\\
\vdots\\
0\\
R_{1}^{(n)}
\end{array}\right]^{T} & \left[\begin{array}{ccc}
L^{(n)}+EE^{T} & \cdots & L^{(n)}\\
\vdots & \ddots & \vdots\\
L^{(n)} & \cdots & L^{(n)}+EE^{T}
\end{array}\right]\left[\begin{array}{c}
R_{1}^{(n)}\\
0\\
\vdots\\
0\\
R_{1}^{(n)}
\end{array}\right]\\
 & =2\left(R_{1}^{(n)}\right)^{T}\left(L^{(n)}+EE^{T}\right)\left(R_{1}^{(n)}\right)+2\left(R_{1}^{(n)}\right)^{T}L^{(n)}\left(R_{1}^{(n)}\right)\\
 & =4\left(R_{1}^{(n)}\right)^{T}L^{(n)}\left(R_{1}^{(n)}\right)=4\lambda_{1}^{(n)}.
\end{align*}
Thus, 
\[
\lambda_{1}^{(n+1)}\le2\lambda_{1}^{(n)}.
\]
\par\end{flushleft}

\end{proof}
By a very similar method, it can also be proved, as shown next, that
the second largest eigenvalue $\lambda_{n!-1}^{(n)}$, which is positive,
grows at least exponentially fast.
\begin{proposition}
$\lambda_{n!-1}^{(n)}\ge2^{n-2}\lambda_{1}^{(2)}=2^{n-2}>0$.
\end{proposition}

\begin{proof}
Using the identity
\[
\lambda_{(n+1)!-1}^{(n+1)}=\max_{E^{T}R=0}\frac{R^{T}L^{(n+1)}R}{R^{T}R},
\]
with a particular choice of 
\[
R=\left[\begin{array}{c}
R_{n!-1}^{(n)}\\
0\\
\vdots\\
0\\
R_{n!-1}^{(n)}
\end{array}\right],
\]
where $R_{n!-1}^{(n)}$ is the eigenvector associated with the second
largest eigenvalue $\lambda_{n!-1}^{(n)}$ of $L^{(n)}$, leads to
$\lambda_{(n+1)!-1}^{(n+1)}\ge2\lambda_{n!-1}^{(n)}$ and thus proves
the result.
\end{proof}
The above bounds for $\lambda_{1}^{(n)}$ and $\lambda_{n!-1}^{(n)}$
are far from tight even as far as their asymptotic orders are concerned.
Numerical evidence is collected in the following table:
\begin{center}
\begin{tabular}{|c|c|c||c|c|}
\hline 
$n$ & $\lambda_{1}^{(n)}$ & $\lambda_{1}^{(n+1)}/\lambda_{1}^{(n)}$ & $\lambda_{n!-1}^{(n)}$ & $\lambda_{(n+1)!-1}^{(n+1)}/\lambda_{n!-1}^{(n)}$\tabularnewline
\hline 
\hline 
$4$ & $-2$ & $1$ & $6.6055$ & $1$\tabularnewline
\hline 
$5$ & $-5.0835$ & $2.5417$ & $30.0293$ & $4.5460$\tabularnewline
\hline 
$6$ & $-20.2413$ & $3.9817$ & $166.1372$ & $5.5324$\tabularnewline
\hline 
$7$ & $-102.9541$ & $5.0860$ & $1083.7641$ & $6.5233$\tabularnewline
\hline 
\end{tabular}.
\par\end{center}

A reasonable conjecture will be that both the smallest and the second
largest eigenvalues grow at a factorial-like speed. More precisely,
we believe that
\[
\lim_{n\rightarrow+\infty}\frac{\lambda_{1}^{(n+1)}}{\lambda_{1}^{(n)}(n-1)}=c_{1}\ge1,
\]
and that 
\[
\lim_{n\rightarrow+\infty}\frac{\lambda_{(n+1)!-1}^{(n+1)}}{\lambda_{n!-1}^{(n)}(n+1/2)}=c_{2}\ge1.
\]

\section{Concluding Remarks\label{sec:concluding-remarks}}

The $\sqrt{n}$ lower-bound conjecture of Bukh and Zhou is still open
and seems quite reasonable in view of the fact that $\mathbb{E}LCS(\sigma_{1},\sigma_{2})\sim2\sqrt{n}$,
in case $\sigma_{1}$ is uniform and $\sigma_{2}$ arbitrary (again,
see \cite{houdre2014central}). We do not have a proof of this conjecture,
but let us nevertheless present, next, a quick $\sqrt[3]{n}$ lower
bound result.

We start with a lemma describing a balanced property among the lengths
of the $LCSs$ of pairs of any three arbitrary deterministic permutations.
This result is essentially due to Beame and Huynh-Ngoc (\cite{beame2008value}).
\begin{lemma}
\label{lem:prod}For any $\pi_{i}\in{\cal S}_{n}$ $(i=1,2,3$), 
\[
LCS(\pi_{1},\pi_{2})LCS(\pi_{2},\pi_{3})LCS(\pi_{3},\pi_{1})\ge n.
\]
\end{lemma}

\begin{proof}
The proof of Lemma $5.9$ in \cite{beame2008value} applies here with
slight modification. We further note that this inequality is tight,
since letting $\pi_{1}=\pi_{2}=id$ and $\pi_{3}=rev(id)$, which
is the reversal of the identity permutation gives, $LCS(\pi_{1},\pi_{2})LCS(\pi_{2},\pi_{3})LCS(\pi_{3},\pi_{1})=n$. 
\end{proof}
In Lemma \ref{lem:prod}, taking $(\pi_{1},\pi_{2})=(id,rev(id))$
gives, for any third permutation $\pi_{3}$, $LCS(id,\pi_{3})LCS(rev(id),\pi_{3})\ge n/LCS(id,rev(id))=n$.
But, since $LCS(id,\pi_{3})$ and $LCS(rev(id),\pi_{3})$ are respectively
the lengths of the longest increasing/decreasing subsequences of $\pi_{3}$,
this lemma can be considered to be a generalization of a well-known
classical result of Erd\"{o}s and Szekeres (see \cite{romik2015surprising}).

We are now ready for the cubic root lower bound.
\begin{proposition}
Let $P$ be an arbitrary probability distribution on ${\cal S}_{n}$
and let $\sigma_{1}$ and $\sigma_{2}$ be two i.i.d.$\ $random permutations
sampled from $P$. Then, for any $n\ge1$, $\mathbb{E}_{P}[LCS(\sigma_{1},\sigma_{2})]\ge\sqrt[3]{n}$.
\end{proposition}

\begin{proof}
Let $\pi_{1},\ \pi_{2}$ and $\pi_{3}\in S_{n}$ and set 
$$L(\pi_{i}):=\sum_{\pi_{1}\in{\cal S}_{n}}p(\pi_{1})LCS(\pi_{1},\pi_{i}) \allowbreak =\sum_{\pi_{1}\in{\cal S}_{n}}LCS(\pi_{i},\pi_{1})p(\pi_{1}),$$
for $i=2,3$. Then,
\begin{align}
 & L(\pi_{2})+LCS(\pi_{2},\pi_{3})+L(\pi_{3})\nonumber \\
= & \sum_{\pi_{1}\in S_{n}}p(\pi_{1})(LCS(\pi_{1},\pi_{2})+LCS(\pi_{2},\pi_{3})+LCS(\pi_{3},\pi_{1}))=3\sqrt[3]{n}\sum_{\pi_{1}\in S_{n}}p(\pi_{1})=3\sqrt[3]{n},\label{eq:expectationsumlb1}
\end{align}
by the arithmetic mean-geometric mean inequality and the previous
lemma. Further, summing over $p(\pi_{2})$ in (\ref{eq:expectationsumlb1})
gives:
\begin{align*}
\sum_{\pi_{2}\in{\cal S}_{n}} & p(\pi_{2})(L(\pi_{2})+LCS(\pi_{2},\pi_{3})+L(\pi_{3}))\\
= & \sum_{\pi_{2}\in{\cal S}_{n}}p(\pi_{2})L(\pi_{2})+L(\pi_{3})+L(\pi_{3})\ge3\sqrt[3]{n}.
\end{align*}
Repeating this last procedure but with weights over $p(\pi_{3})$
leads to
\begin{equation}
\sum_{\pi_{2}\in{\cal S}_{n}}p(\pi_{2})L(\pi_{2})+2\sum_{\pi_{3}\in{\cal S}_{n}}p(\pi_{3})L(\pi_{3})=3\sum_{\pi\in S_{n}}p(\pi)L(\pi)\ge3\sqrt[3]{n}.\label{eq:expectsumlower2}
\end{equation}
However,
\begin{align*}
\mathbb{E}_{P}[LCS(\sigma_{1},\sigma_{2})] & =\sum_{\pi_{1}\in{\cal S}_{n}}\sum_{\pi_{2}\in{\cal S}_{n}}p(\pi_{1})LCS(\pi_{1},\pi_{2})p(\pi_{2})\\
 & =\sum_{\pi_{1}\in{\cal S}_{n}}p(\pi_{1})\sum_{\pi_{2}\in{\cal S}_{n}}LCS(\pi_{1},\pi_{2})p(\pi_{2})\\
 & =\sum_{\pi\in{\cal S}_{n}}p(\pi)L(\pi).
\end{align*}
Combining this last identity with (\ref{eq:expectsumlower2}) proves
the result.
\end{proof}
The above proof is simple; it basically averages out each $LCS(\cdot,\cdot)$
as $\sqrt[3]{n}$ on the summation weighted by $P$. However, in view
of the original conjecture, our partial results, as well as those mentioned in the introductory section, the cubic root
lower-bound is not tight. Apart from our curiosity concerning this
$\sqrt{n}$ conjecture, it would be interesting to know the exact
asymptotic order of the smallest eigenvalue $\lambda_{1}^{(n)}$ of
$L^{(n)}$. In contrast, the largest eigenvalue $\lambda_{n!}^{(n)}$
corresponding to the uniform distribution is known to be asymptotically
of order $2n!\sqrt{n}$, since it is equal to the length of the $LISs$
of a uniform random permutation of $[n]$ scaled by $n!$. In this
sense, the study of the length of the $LCSs$ between a pair of i.i.d.$\ $random
permutations having an arbitrary distribution, or equivalently, the
study of $L^{(n)}$, can be viewed as an extension of the study of
the length of the $LISs$ of a uniform random permutation of $[n]$.
Having a complete knowledge of the distribution of all the eigenvalues
of $L^{(n)}$ would be a nice achievement. 


\end{document}